\documentclass[a4paper,11pt]{article}
\usepackage[english]{babel}
\usepackage[T1]{fontenc}
\usepackage[latin1]{inputenc}
\usepackage{amsmath,amssymb,mathrsfs}
\usepackage{soul}
\usepackage{xypic}
\usepackage{amsthm}

\theoremstyle{plain}
\newtheorem{thm}{Theorem}[section]

\theoremstyle{definition}
\newtheorem{defi}[thm]{Definition}
\newtheorem{rem}[thm]{Remark}
\newtheorem{exm}[thm]{Example}
\newtheorem{cons}[thm]{Consequence}

\title{The concept of orthogonality in Cartan's geometry based on the concept of
area}
\author{Imsatfia Moheddine\footnote{Institut de Mathématiques de Jussieu, UMR 7586, Université Paris Diderot-Paris 7, Site de Chevaleret, 75013 Paris, France.}}

\begin{document}

\maketitle

\begin{abstract}
In 1931 Elie Cartan constructed a geometry which was rarely considered.
Cartan proposed a way to define an infinitesimal metric $ds$ starting from a variational problem on
hypersurfaces in an $n$-dimensional manifold $\mathcal{M}$. This distance depends
not only of the point $\textsc{m}\in\mathcal{M}$ but on the orientation of a hyperplane in
the tangent space $T_{\textsc{m}}\mathcal{M}$. His first step is a natural
definition of the orthogonal direction to such tangent hyperplane. In this
paper we extend it, starting form considerations from the calculus of variation.
\end{abstract}
\begin{LARGE}
\textbf{Introduction}\\
\end{LARGE}

Riemann considered the possibility to give to $ds$, the distance between two
infinitesimally close points, a much more general expression than
$\sqrt{g_{\imath\jmath}(dx^\imath,dx^\jmath)}$ namely to choose any function of
$x$ and $dx$ which is homogeneous of degree 1 in $dx$. This more general
geometry was later developped by P. Finsler, E. Cartan, \cite{Cartan1934}, more recently, by Chern, H.Rund and Bryant \cite{bryant2002}... .

In \cite{Cartan1933} Cartan proposed another generalisation of Riemannian geometry where the distance
between two infinitesimally closed points in $\mathcal{M}$ depends of the point $\textsc{m}$
\underline{and} of the choice of a hyperplane in the tangent space to
the manifold. In the modern language, this amounts to define a metric on the
vector bundle over the Grassmannian bundle of oriented hyperplanes,
$Gr_{n-1}(\mathcal{M})$ whose fiber at $\textsc{m}\in E$ is the set of oriented hyperplanes in $T_{\textsc{m}}\mathcal{M}$ (where $\textsc{m}\in\mathcal{M}$ and $E$ called \textit{``element''} by Cartan, denotes an oriented hyperplane in $T_{\textsc{m}}\mathcal{M}$). Moreover Cartan
found a way to canonically derive such a metric from a variational problem on
hypersurfaces in $\mathcal{M}$. He simultaneously defined a connection on this
bundle. The first step consists in choosing a natural definition for the orthogonal complement of an
element $E$ and the metric in the normal direction: The idea is to require that,
for any extremal hypersurface $\mathcal{H}$ of $\mathcal{M}$ and any compact
subset with a smooth boundary $\Sigma\in \mathcal{H}$ if we perform a
deformation of $\partial\Sigma$ in the \emph{normal direction} to $\Sigma$ and with an arbitrary indensity and consider
the family of extremal hypersurfaces spanned by the images of $\partial\Sigma$
by this deformations, then the area of hypersurfaces is stationary. This uses a
formula of De Donder (which is basically an extension to variational problems
with several variables of a basic formula in the theory of integral
invariants). Let us now present this idea for submanifolds of arbitrary codimension $n-p$. Such variational problem can be described as follows. Let $\beta$ be a $p$-form 
which, in local coordinates $x^1,...,x^n$, reads $\beta=dx^1\wedge...\wedge dx^p$. Any $p$-dimensional oriented submanifold $\mathcal{N}$ such that $\beta\vert_{\mathcal{N}}
>0$ can be locally represented as the graph of a function $f=(f^1,...,f^{n-p})$ of the variables $(x^1,...,x^p)$. We consider functional $\mathcal{L}$ of the form 
$\mathcal{L}(f):=\int_{\mathcal{N}} d\sigma$, when
\[
d\sigma=L\left(x^1,...,x^p,f^1,..f^{n-p},\nabla f\right)\beta.
\]
Let $\mathcal{N}$ the critical point of $\mathcal{L}$. To define the
orthogonal subspaces to all tangent subspaces to $\mathcal{N}$ the idea is to consider a $1$-parameter family $(\mathcal{N}_t)_t$ of submanifolds which form a foliation
of a submanifold $U$ of dimension $p+1$ in $\mathcal{M}$ and such that $\mathcal{N}_0=\mathcal{N}$. Consider a vector field $X$ on $U$ witch induces the variation from $\mathcal{N}_t$ to
$\mathcal{N}_{t+dt}$ and denote
\[
\mathcal{A}(t)=\mathcal{L}(f_t).
\]
According to Cartan \cite{Cartan1933} the condition for $X$ to be orthogonal to
$\mathcal{N}=\mathcal{N}_0$ is that the derivative of $\mathcal{A}(t)$ with
respect to $t$ at $t=0$ is zero. This will allow us to find the orthogonal
subspace. We will show that this definition of $E^{\perp}$ actually dose not depend on the choice of $\mathcal{N}$ but uniquelly on $E\in Gr_\textsc{m}^p\mathcal{M}$.

\section{Cartan geometry based on the concept of area}

Let $\mathcal{M}$ be a manifold of $n$-dimensional then we define the \emph{Grassmannian bundle} or \emph{Grassmannian} by
\[
Gr_p\mathcal{M}=\{(\textsc{m},E)\vert \textsc{m}\in \mathcal{M}; E \text{ an oriented $p$-dimensional }
\]
\[
\text{ vector subspace in }T_\textsc{m}\mathcal{M}\}.\ \ \ \ \ \ \ \ \ \ \ \ 
\]
And, if $\beta$ is a $p$-form which in local coordinates $(x^1,...,x^n)$, reads $\beta=dx^1\wedge...\wedge dx^p$ where $1\leq p\leq n-1$, then
\[
Gr_p^\beta\mathcal{M}=\{(\textsc{m},E)\in T_\textsc{m}\mathcal{M}\vert \beta=dx^1\wedge...\wedge
dx^p\vert_E>0\}.
\]
Let $(q^\jmath)_{1\leq\jmath\leq p(n-p)}$ be coordinate fonctions on $Gr_p^\beta\mathcal{M}$ such that 
$(x^\imath,q^\jmath)$ are local coordinates on $Gr_p^\beta\mathcal{M}$. We denote the projection $ \pi $ by:
\[
\pi:Gr_p^\beta\mathcal{M} \longrightarrow \mathcal{M}
\]
\[
\ \ \ \ (\textsc{m},E)\longmapsto \textsc{m}
\]
We consider $\pi^\ast T\mathcal{M}$ the bundle over the Grassmannian whose
fiber at $(\textsc{m},E)$ is $T_\textsc{m}\mathcal{M}$, we denote a metric $g$ on $\pi^\ast
T\mathcal{M}$ by

\[
g_{(\textsc{m},E)}=g_{\imath \jmath}(x^k,q^k)dx^\imath dx^\jmath,
\]
We see that the coefficients $g_{\imath\jmath}$ not only depend on coordinates
of $\textsc{m}$, but they also depend on the orientation of the element
at $\textsc{m}$.
\begin{rem}
If $p=n-1$ then
\[
Gr_{n-1}(\mathcal{M})\sim (T^\ast\mathcal{M}\setminus\{0\})\diagup\mathbb{R}^\ast.
\]
\end{rem}

\begin{defi}
A \textit{geometry based on the concept of area} $(\mathcal{M},F)$ is a
differential manifold $\mathcal{M}$ equipped with a function $F$ defined over
$T^\ast\mathcal{M}$ with values in $\mathbb{R}+$
\[
F:T^\ast\mathcal{M}\rightarrow \mathbb{R}+,
\]

which satisfies the following conditions:
\begin{enumerate}
 \item $F$ is $\mathcal{C}^\infty$ over
$T^\ast\mathcal{M}\setminus\{0\}:=\bigcup_{\textsc{m}\in\mathcal {M}}
T_\textsc{m}^\ast\mathcal{M}\setminus\{0\}$.
\item $F$ is homogeneous of degree one in $q^k$
\[
F(x^k,\lambda q^k)= \lambda F(x^k,q^k).
\]
\item The Hessian matrix defined by
\[
(g_{\imath \jmath}):=[\frac{1}{2}(F^2)_{q^\imath q^\jmath}]
\]
is positive definite at any point of $Gr_p(\mathcal{M})$.
\end{enumerate}
In other words, $F\mid_{T^\ast_\textsc{m}\mathcal{M}}$ is a Minkowski norm for all $\textsc{m}\in\mathcal{M}$.
\end{defi}

\section{The concept of orthogonality in the space of Cartan}
In the following, since we work locally we shall identity $\mathcal{M}$ with $\mathbb{R}^n$ to the coordinate system $(x^\imath)_\imath$.
\subsection{Lagrangian formulation}
Let $L:Gr_p^\beta(\mathbb{R}^p\times\mathbb{R}^{n-p}):=\{(x^1,...,x^n,(q^\imath_\jmath)
_{ \substack{ 1\leq\imath\leq n-p \\ 1\leq\jmath\leq p}})\longrightarrow\mathbb{R}\}$ be the Lagrangian function.
For any function $f:\Omega\subset\mathbb{R}^p\rightarrow\mathbb{R}^{n-p}$ of class
$\mathcal{C}^\infty$, we denote by $\Gamma_f$ its graph. A point $x\in\Gamma_f$ is defined by $(x^{p+1},...,x^n)=(f^1(x^1,...,x^p),...,f^{n-p}(x^1,...,x^p))$ and values of
the coordinates $(q^\imath_\jmath)$ at the tangent space to $\Gamma_f$ are given by $(\nabla f)(x)$. Let $\beta=dx^1\wedge...\wedge dx^p$ be a $p$-form, the action 
integral \cite{Helein2004} is given by
\[
\mathcal{L}(f)=\int_{\Omega} L(x^1,...,x^p,f^1, ...,f^{n-p},\nabla f)\beta=\int_{\Omega} L(x,f,\nabla f)\beta.
\]
The bundle over the Grassmannian of $\Gamma_f$ given by
\[
Gr_p^\beta(\Gamma_f):=\{(x,E);
x\in\Gamma_f,\ \ E = T_x\Gamma_f\}.
\]
\begin{defi}
Let $\Gamma$ be an oriented $p$-dimensional submanifold of $\mathcal{M}$ with boundary $\Gamma_0$ which is a critical point of $\mathcal{L}$. A distribution of vector lines $\mathcal{D}$ in $T\mathcal{M}$ along $\Gamma_0$ is called \emph{normal} if, for any vector field $N$ defined along $\Gamma_0$ such that $\forall \textsc{m}\in\Gamma_0, \ \ N(\textsc{m})\in
\mathcal{D}(\textsc{m})$, and if $\partial\Gamma_t:=\{e^{tN}(\textsc{m})\vert \textsc{m}\in\partial
\Gamma,t\in(-\varepsilon,\varepsilon)\}$ and $\mathcal{A}(t):=\mathcal{L}(\Gamma_t)$ then
$\frac{d}{dt}\left(\mathcal{A}(t)\right)\vert_{t=0}=0$.
\end{defi}
\begin{thm}
There exists a vector subbundle $\pi^\ast T^{\perp}\mathcal{M}$ of $\pi^\ast T\mathcal{M}$ of rank $n-p$ whose fiber at $(x,E)$ is denoted by $(\pi^\ast T^{\perp}\mathcal{M})_{(x,E)}$ such that for any oriented $p$-dimensional critical point $\Gamma$ of $\mathcal{L}$, a vector field $N$ along $\partial\Gamma$ is normal if and only if $N_x\in(\pi^\ast T^{\perp}\mathcal{M})_{(x,T_x\Gamma)}$. In the following we write $(\pi^\ast T^{\perp}\mathcal{M})_{(x,T_x\Gamma)}=(T_x\Gamma)^\perp$. Moreover $(T_x\Gamma)^\perp$ is spanned by $(v^1,...v^{n-p})$, where
\[
v^1=\left(
\begin{array}
l\frac{\partial L}{\partial q^1_1}\\
\vdots\\
\frac{\partial L}{\partial q^1_p}\\
-L+q^1_\jmath\frac{\partial L}{\partial q^1_\jmath}\\
0\\
0\\
\vdots\\
0\\
\end{array}
\right),
v^2=\left(
\begin{array}
l\frac{\partial L}{\partial q^2_1}\\
\vdots\\
\frac{\partial L}{\partial q^2_p}\\
0\\
-L+q^2_\jmath\frac{\partial L}{\partial q^2_\jmath}\\
0\\
\vdots\\
0\\
\end{array}
\right),
...,v^{n-p}=\left(
\begin{array}
l\frac{\partial L}{\partial q^{n-p}_1}\\
\vdots\\
\frac{\partial L}{\partial q^{n-p}_p}\\
0\\
0\\
\vdots\\
0\\
-L+q^{n-p}_\jmath\frac{\partial L}{\partial q^{n-p}_\jmath}\\
\end{array}
\right)
\]
\end{thm}
\begin{proof}
\textit{Consider first the case $p=2$ and $n=3$}\\ The Grassmaniann of dimension 5, the Lagrangian $(x,y,z,p,q)\mapsto L(x,y,z,p,q):=L(x,y,f(x,y),\nabla
f(x,y))$ and the action integral is given by
\[
\mathcal{L}(f)=\int_{\mathbb{R}^2} L(x,y,f(x,y),\nabla f(x,y))\beta
\]
Suppose that this integral is extended to a portion of extremal surface
$\Sigma$ limited by a contour $\mathcal{C}$, deform slightly $\Sigma$ to a surface
$\Sigma'$ limited by a contour $\mathcal{C}'$. This amounts to change in the
preceding integral $f$ into $f+\varepsilon g$ where $g$ has not necessarily a compact support.
Then we consider a family $(\Sigma_t)_t$ of surfaces with boundary which forms a
foliation of a domain $U\subset\mathbb{R}^3$ which coincides in $t=0$ with
$\Sigma$ and in $t=1$ with $\Sigma'$, depending of a real parameter $t\in[0,1]$.
We suppose that for all $t$, $(\Sigma_t)_t$ is a critical point of $\mathcal{L}$ that we will represent by the graph $\Sigma_t$ of a function $f_t:\Omega_t\rightarrow \mathbb{R}$
\[
\Sigma_t=\left\{(x,y,f_t(x,y))\backslash (x,y)\in \Omega_t\right\}
\]

Let $X$ a vector field defined on $U$ such that, if $e^{sX}$ is the flow of $u$, then:
\[
e^{sX}(\Sigma_t)=\Sigma_{t+s}.
\]
Note
\[
\left\{
\begin{array}
lf(t,x,y)=f_t(x,y)\\
f(x,y)=f(0,x,y)=f_0(x,y)\\
\Phi(t,x,y)=e^{tX}(x,y,f(x,y))
\end{array}
\right.
\]
if $t=0$ we have $\Phi=f=f_0$ and $\forall t\in ]0,1],$ the function
$(x,y)\mapsto\Phi(t,x,y)$ is a parametrization of $\Sigma_t$, we denote by
$\Phi=(\phi^1,\phi^2,\phi^3)$ and $\phi^3(t,x,y)=f(t,\phi^1(t,x,y),\phi^2(t,x,y))$ so, if we derive with respect to $t$ then
\[
\frac{\partial\phi^3}{\partial t}=\frac{\partial f}{\partial
t}(t,\phi^1,\phi^2)+\frac{\partial f}{\partial x}(t,\phi^1,\phi^2)\frac{\partial
\phi^1}{\partial t}+\frac{\partial f}{\partial y}(t,\phi^1,\phi^2)\frac{\partial
\phi^2}{\partial t}.
\]
which gives for $t=0$
\[
X^3(x,y,f)=\frac{\partial f}{\partial t}(0,x,y)+\frac{\partial f}{\partial
x}(0,x,y)X^1(x,y,f) +\frac{\partial f}{\partial y}(0,x,y)X^2(x,y,f).
\]
Thus along $\Sigma=\Sigma_0$, we have:
\begin{equation}
\frac{\partial f}{\partial t}=X^3-X^1\frac{\partial f}{\partial
x}-X^2\frac{\partial f}{\partial y}
\end{equation}
Let the Lagrangian $(x,y,z,p,q)\mapsto L(x,y,z,p,q)$ and we consider
\[
\mathcal{A}(t)=\int_{\Omega_t}L(x,y,f_t(x,y),\frac{\partial f_t}{\partial
x}(x,y),\frac{\partial f_t} {\partial y}(x,y))dxdy
\]
Assuming that $\Omega_t$ is regular (ie $\partial\Omega_t$ is a curve
$\mathcal{C}^1$ of plan $\mathbb{R}^2$) then we have
\[
\frac{d\mathcal{A}(t)}{dt}=\int_{\Omega_t}\frac{\partial}{\partial
t}L(x,y,f_t,\frac{\partial f_t}{\partial x},\frac{\partial f_t} {\partial
y})dxdy+\int_{\partial\Omega_t}L(x,y,f_t,\frac{\partial f_t}{\partial
x},\frac{\partial f_t} {\partial y})\langle(X^1,X^2),\nu\rangle d\ell
\]
Where $\nu$ is a exterior normal of $\Omega_t$ in $\mathbb{R}^2$ and
$\langle(X^1,X^2),\nu\rangle$ is the change in the area of $\Omega_t$ and
$d\ell$ is a measure of one dimension $\partial\Omega$, thus

\[
\frac{d\mathcal{A}(t)}{dt}\mid_{t=0}=\int_{\Omega_0}\frac{\partial}{\partial
t}L(x,y,f_t,\frac{\partial f_t}{\partial x},\frac{\partial f_t} {\partial
y})dxdy+\int_{\partial\Omega_0}L(x,y,f_t,\frac{\partial f_t}{\partial
x},\frac{\partial f_t} {\partial y})\langle(X^1,X^2),\nu\rangle d\ell
\]
\[
=\int_{\Omega}\frac{\partial L}{\partial z}\frac{\partial f}{\partial
t}+\frac{\partial L}{\partial p} \frac{\partial^2f}{\partial x\partial
t}+\frac{\partial L}{\partial q} \frac{\partial^2f}{\partial y\partial
t}+\int_{\partial\Omega}L\langle(X^1,X^2),\nu\rangle d\ell
\]
\[
=\int_\Omega\frac{\partial L}{\partial z}\frac{\partial f}{\partial
t}+\frac{\partial}{\partial x}\left (\frac{\partial L}{\partial p}\frac{\partial
f}{\partial t}\right)+\frac{\partial}{\partial y}\left( \frac{\partial
L}{\partial q}\frac{\partial f}{\partial t}\right)-\frac{\partial f}{\partial
t}\left[ \frac{\partial}{\partial x}\left(\frac{\partial L}{\partial p}\right)
+\frac{\partial}{\partial y}\left( \frac{\partial L}{\partial q}\right)\right]
\]
\[
+\int_{\partial\Omega}L\langle(X^1,X^2),\nu\rangle d\ell.
\]

\[
=\int_{\Omega}\frac{\partial f}{\partial t}\left[\frac{\partial L}{\partial
z}-\frac {\partial}{\partial x}\left(\frac{\partial L}{\partial
p}\right)-\frac{\partial}{\partial y}\left( \frac{\partial L}{\partial
q}\right)\right]+\int_{\partial\Omega} \left\langle  \left (\frac{\partial
L}{\partial p}\frac{\partial f}{\partial t},\frac{\partial L}{\partial q}
\frac{\partial f}{\partial t}\right),\nu\right\rangle d\ell
\]
\[
+\int_{\partial\Omega}L\langle(X^1,X^2),\nu\rangle d\ell.
\]
\[
=\int_{\Omega}\frac{\partial f}{\partial t}\left[\frac{\partial L}{\partial z}
-\frac{\partial}{\partial x}\left(\frac{\partial L}{\partial p}\right)
-\frac{\partial}{\partial y}\left(\frac{\partial L}{\partial q}\right)\right]
+\int_{\partial\Omega}\left\langle\left(\frac{\partial L}{\partial
p}\frac{\partial f}{\partial t}+LX^1,\frac{\partial L}{\partial q}\frac{\partial
f}{\partial t}+LX^2\right),\nu\right\rangle d\ell
\]
Or we know that $\Sigma=\Sigma_0$ is a critical point of $\int_\Omega L$, then
the Euler Lagrange equations are verified, thus
\[
\frac{d\mathcal{A}(t)}{dt}\mid_{t=0}=\int_{\partial\Omega}\left\langle\left(\frac{\partial
L}{\partial p}\frac{\partial f}{\partial t}+LX^1,\frac{\partial L}{\partial
q}\frac{\partial f}{\partial t}+LX^2\right),\nu\right\rangle dl
\]
We now assum that $X\vert_{\partial\Gamma}$ has the form $\psi N_0\in\mathcal{D}$ where $\psi\in\mathcal{C}^\infty (\partial\Gamma)$ with values in $\mathbb{R}$ and where $N_0$ is a fixed non-vanishing tangent defined along $\partial\Gamma$, to be determined we seek a condition to $N_0$ such that for any regular function $\psi$, $ \frac{d\mathcal{A}(t)}{dt}\mid_{t=0}=0$. We can choose a function $f_t$ depends of $\psi$ such that $\frac{\partial f_{t,\psi}}{\partial t}\mid_{t=0}=\psi\frac {\partial f_t}{\partial t}\mid_{t=0}$ we replace this in the previous integral, thus
\[
\frac{d\mathcal{A}(t)}{dt}\mid_{t=0}=\int_{\partial\Gamma}\left\langle\left(\frac{\partial
L}{\partial p}\psi\frac{\partial f}{\partial t}+\psi LN_0^1,\psi\frac{\partial
L}{\partial q}\frac{\partial f}{\partial t}+\psi LN_0^2\right),\nu\right\rangle dl\ \ \ \ \ \ \ \
\]
\[
=\int_{\partial\Gamma}\psi\left\langle\left(\frac{\partial L}{\partial
p}\frac{\partial f}{\partial t}+LN_0^1,\frac{\partial L}{\partial q}\frac{\partial
f}{\partial t}+ LN_0^2\right),\nu\right\rangle dl.
\]
The condition for $\frac{d\mathcal{A}(t)}{dt}\mid_{t=0}=0$ for all $\psi$
regular function on $\partial\Gamma$ and $\nu$ exterior normal of $\Gamma$ is
that $\left\langle\left(\frac{\partial L}{\partial p}\frac{\partial f}{\partial
t}+LN_0^1,\frac{\partial L}{\partial q}\frac{\partial f}{\partial t}+ LN_0^2\right), \nu\right\rangle=0$. If we denote by $\lambda=\frac{-\frac{\partial f}{\partial t}}{L}$, then this enthusiasm that
\[
\left\{
\begin{array}{c}
N_0^1=\lambda\frac{\partial L}{\partial p}\\
\\
N_0^2=\lambda\frac{\partial L}{\partial q}
\end{array}
\right.
\]
Form (1), we have
\[
N_0^3=-\lambda L+\lambda\frac{\partial L}{\partial p}\frac{\partial f}{\partial
x}+\lambda\frac{\partial L}{\partial q}\frac{\partial f}{\partial y}.
\]
Thus, so that $\frac{d\mathcal{A}(t)}{dt}\mid_{t=0}=0$ it suffices that
\[
X=\left(\frac{\partial L}{\partial p},\frac{\partial L}{\partial
q},p\frac{\partial L}{\partial p}+q\frac{\partial L}{\partial q}-L\right).
\]

\textit{\textbf{Lets now $n>0$ and $p=n-1$}}\\ The grassmanian bundle is of dimension $2n-1$.
It is the same as previous, thus the orthogonal of $\pi^\ast T\mathcal{M}$ of rank $n-1$ is spanned by
\begin{equation}
X=\left(\frac{\partial L}{\partial p^1},...,\frac{\partial L}{\partial
p^{n-1}}, \sum_{\imath=1}^{n-1}p^\imath\frac{\partial L}{\partial
p^\imath}-L\right),
\end{equation}
where $p^\imath=\frac{\partial f}{\partial x_\imath}$ for $\imath=1,...n-1$, and $L$ be the Lagrangian on $Gr_{n-1}(\Sigma)$.\\

\textit{\textbf{Now the case $n>3$ and $p<n$}}\\
For $1\leq p\leq n-1$ let $\Omega$ be a regular open set of $\mathbb{R}^p$ and $f=(f^1,... f^{n-p}):\Omega\longrightarrow\mathbb{R}^{n-p}$, we denote its graph by
\[
\mathcal{N}:=\{(x,f(x))\mid x\in\Omega\}.
\]
Let $\beta=dx^1\wedge...\wedge dx^p$ be a $p$-form, and $L$ be the Lagrangian on $Gr^\beta_p(\mathcal{N})$, thus the action integral is given by
\[
\mathcal{L}(f):=\int_{\mathcal{N}}L(x^1,...,x^p,f^1,...,f^{n-p},\nabla f)\beta.
\]
The family $(\mathcal{N}_t)_t$ of submanifolds of board form a foliation in
a submanifold $U\subset\mathbb{R}^{p+1}$, we suppose that for all $t$, $\mathcal{N}_t$ is a
critical point of $\mathcal{L}$. Let $X$ be a vector field defined on $U$ such that, if $e^{sX}$ is the flow of $u$, then
$e^{sX}(\mathcal{N}_t)=\mathcal{N}_{t+s}$, denote
\[
\left\{
\begin{array}
lf(t,x^1,...,x^p)=f_t(x^1,...x^p)\Leftrightarrow
f^\imath(t,x^1,...,x^p)=(f^\imath) _t(x^1,...x^p)\forall \imath=1,...n-p,\\
\\
f(x^1,...,x^p)=f(0,x^1,...,x^p)=f_0(x^1,...x^p)\Leftrightarrow \forall
\imath=1,...n-p \text { we have }\\
f^\imath(x^1,...,x^p) =f^\imath(0,x^1,...,x^p)=(f^\imath)_0(x^1,...x^p),\\
\\
\Phi(t,x^1,...,x^p)=e^{tX}(x^1,...,x^p,f^1,...,f^{n-p}),
\end{array}
\right.
\]
The function $\Phi$ is the parametrization of $\mathcal{N}_t$, we denote:
\[
\left\{
\begin{array}
l \Phi=(\varphi^1,...,\varphi^p,\varphi^{p+1},...,\varphi^n),\\
\varphi^{p+\imath}(t,x^1,...,x^p)=f^\imath(t,\varphi^1,...,\varphi^p) \text{
for } \imath=1,...n-p,
\end{array}
\right.
\]
Thus, $\forall \imath=1,...n-p$:
\[
\frac{\partial \varphi^{p+\imath}}{\partial t}=\frac{\partial
f^\imath}{\partial t}
(t,\varphi^1,...,\varphi^p)+\sum_{\jmath=1}^p\frac{\partial f^\imath}{\partial
x^\jmath} (t,\varphi^1,...,\varphi^p) \frac{\partial \varphi^\jmath}{\partial
t}.
\]
For $t=0$, $\forall \imath=1,...n-p$, thus
\[
X^{p+\imath}(x,f)=\frac{\partial f^\imath}{\partial t}
(0,x^1,...,x^p) +\sum_{\jmath=1}^p\frac{\partial f^\imath}{\partial x^\jmath}
(0,x^1,...,x^p) X^\jmath(x,f),
\]
this given along $\mathcal{N}=\mathcal{N}_0$ and $\forall \imath=1,...n-p$
\begin{equation}
\frac{\partial f_\imath}{\partial
t}=X^{p+\imath}-\sum_{\jmath=1}^p\frac{\partial f_\imath}{\partial
x^\jmath}X^\jmath.
\end{equation}
We have
\[
\mathcal{A}(t)=\mathcal{L}(f_t)=\int_{\mathcal{N}_t}L(x^1,...,x^p,(f^1)_t,...,(f^{n-p})_t,\nabla
f_t)\beta,
\]
thus
\[
\frac{d\mathcal{A}(t)}{\partial t}=\int_{\mathcal{N}_t}\frac{\partial}{\partial t}
L(x^1,...,x^p,(f^1)_t,...,(f^{n-p})_t, \nabla f_t)\beta\ \ \ \ \ \ \ \ \ \ \ \ \ \ \ \ \ \ \ \ \ \ \ \ \ \ \ \ \ \ \ \
\]
\[
+\int_{\partial\mathcal{N}_t}L(x^1,...,x^p,(f^1)_t,...,(f^{n-p})_t,\nabla f_t)
\langle(X^1,...,X^p),\nu\rangle d\ell,
\]
where $\nu$ is the exterior normal to $\mathcal{N}_t$ in $\mathbb{R}^n$,
$\langle(X^1,...,X^p) ,\nu\rangle$ represents the change in volume of $\mathcal{N}_t$
and $d\ell$ is a measure of one dimension $\partial\mathcal{N}_t$. Thus for $t=0$ we have
\[
\frac{d\mathcal{A}(t)}{\partial
t}\mid_{t=0}=\int_{\mathcal{N}_t}\frac{\partial}{\partial t}
L(x^1,...,x^p,(f^1)_t,...,(f^{n-p})_t, \nabla f_t)\beta
+\int_{\partial\mathcal{N}_t}L\langle(X^1,...,X^p),\nu\rangle d\ell.
\]
We calculate $\frac{d\mathcal{A}(t)}{\partial t}\mid_{t=0}$
\[
\int_{\mathcal{N}}\frac{\partial}{\partial t} L(x,f_t,\nabla f_t)\beta
=\int_{\mathcal{N}}\left(\sum_{\imath=1}^{n-p}\frac{\partial L}{\partial
x^\imath}\frac{\partial f^\imath} {\partial t}+\sum_{\substack{1\leq\jmath\leq p
\\ 1<\imath\leq n-p}}\frac{\partial F}{\partial p^\imath_
\jmath}\frac{\partial^2f^\imath}{\partial x^\jmath\partial t}\right)\beta \ \ \ \ \ \ \ \ \ \ \ \ \ \ \ \ \ \ \ \ \ \ \ \ \ \ \ \ \ \ \
\]
\[
=\sum_{\imath=1}^{n-p}\int_{\mathcal{N}}\left(\frac{\partial L}{\partial x^\imath}
\frac{\partial f^\imath} {\partial t}+\sum_{\jmath=1}^p\frac{\partial
F}{\partial p^\imath_\jmath} \frac{\partial^2f^\imath}{\partial x^\jmath\partial
t}\right)\beta \ \ \ \ \ \ \ \ \ \ \ \ \ \ \ \ \ \ \ \ \ \ \ \ \ \ \ \ \ \ \  \ \ \ \ \ \ \ \ \ \ \ \ \ \ \ \ \ \ \ \ \ \ \ \ \ \ \ \ \ \ \
\]
\[
=\sum_{\imath=1}^{n-p}\int_{\mathcal{N}}\left(\frac{\partial L}{\partial x^\imath}
\frac{\partial f^\imath} {\partial t}+\sum_{\jmath=1}^p\frac{\partial}{\partial
x^\jmath} \left(\frac{\partial L}{\partial q^\imath_\jmath}\frac{\partial
f^\imath} {\partial t}\right)-\frac{\partial f^\imath}{\partial
t}\sum_{\jmath=1}^p\frac{\partial}{\partial x^\jmath} \left(\frac{\partial
L}{\partial q^\imath_\jmath}\right)\right) \beta \ \ \ \ \ \ \ \ \ \ \ \ \ \ \ \ \ \ \ \ \ \ \ \ \ \ \ \ \ \ \
\]
\[
=\sum_{\imath=1}^{n-p}\int_{\mathcal{N}}\frac{\partial f^\imath}{\partial
t}\left[\frac{\partial L}{\partial
x^\imath}-\sum_{\jmath=1}^p\frac{\partial}{\partial x^\jmath}\left(\frac
{\partial L}{\partial q^\imath_\jmath}\right)\right]\beta+\sum_{\imath=1}^{n-p}
\int_{\partial{\mathcal{N}}}\left\langle \left(\frac{\partial L}{\partial
q^\imath_1}\frac{\partial f^\imath} {\partial t},..., \frac{\partial L}{\partial
q^\imath_p}\frac{\partial f^\imath} {\partial t}\right),\nu\right \rangle d\ell
\]
We have $\mathcal{N}=\mathcal{N}_0$ is a critical point of $\mathcal{L}$, thus the
Euler Lagrange equations are verified $\frac{\partial L}{\partial
x^\imath}-\sum_{\jmath=1}^p\frac{\partial}{\partial x^\jmath}\left(\frac
{\partial L}{\partial q^\imath_\jmath}\right)=0$ this gives
\[
  \frac{d\mathcal{A}(t)}{\partial t}\mid_{t=0}=\sum_{\imath=1}^{n-p}
\int_{\partial\Omega} \left\langle \left(\frac{\partial L}{\partial
q^\imath_1}\frac{\partial f_\imath} {\partial t}+LX^1,..., \frac{\partial
L}{\partial q^\imath_p}\frac{\partial f_\imath} {\partial
t}+LX^p\right),\nu\right\rangle d\ell
\]
Using the previous definition we can also consider a regular function $\psi:\partial\mathcal{N}\rightarrow\mathbb{R}$ such that $f_{t,\psi}$ such
that $\frac{\partial f^\imath_{t,\psi^i}}{\partial t}=\psi^i\frac{\partial f^\imath_t}{\partial
t}$. By the same as the previous, so that
$\frac{d\mathcal{A}(t)}{\partial t}\mid_{t=0} =0$, it suffices that for all
$\jmath=1,...p$ we have
\[
\sum_{\imath=1}^{n-p}\frac{\partial L}{\partial q^\imath_\jmath}\frac{\partial
f^\imath}{\partial t}+LX^\jmath=0
\]
if we denote $\lambda_\imath=\psi^i\frac{-\frac{\partial f^\imath}{\partial t}}{L}$ and $\nabla f:=\left(\frac{\partial f^\imath}{\partial
x^\jmath}\right)_{\substack {1<\imath\leq n-p \\ 1\leq\jmath\leq p}}=(q^\imath_\jmath)_{\substack{1<\imath\leq n-p \\ 1\leq\jmath\leq p}}$, then
\[
X^\jmath=\sum_{\imath=1}^{n-p}\lambda_\imath\frac{\partial L}{\partial
q^\imath_\jmath} \ \ \ \ \text{ for all }\jmath=1,...p,
\]
from (3), for $\imath=1,...n-p$ thus
\[
X^{p+\imath}=-\lambda_\imath L+\sum_{\jmath=1}^p\lambda_\imath
q^\imath_\jmath\frac{\partial L}{\partial q^\imath_\jmath},
\]
wich gives $X=(X^1,...,X^n)$ with,
\[
\left\{
\begin{array}
l X^1=\lambda_1\frac{\partial L}{\partial q^1_1}+ \lambda_2\frac{\partial
L}{\partial q^2_1}+...+\lambda_{n-p}\frac{\partial L}{\partial q^{n-p}_1}\\
\vdots\\
X^p=\lambda_1\frac{\partial L}{\partial q^1_p}+\lambda_2\frac{\partial
L}{\partial q^2_p}+...+\lambda_{n-p}\frac{\partial L}{\partial q^{n-p}_p}\\
X^{p+1}=\lambda_1\left(-L+q^1_1\frac{\partial L}{\partial
q^1_1}+...+q^1_p\frac{\partial L}{\partial q^1_p}\right)\\
\vdots\\
X^n=\lambda_{n-p}\left(-L+q^{n-p}_1\frac{\partial L}{\partial
q^{n-p}_11}+...+q^{n-p}_p\frac{\partial L}{\partial q^{n-p}_p}\right)
\end{array}
\right.
\]
\[
=\lambda_1\left(
\begin{array}
l\frac{\partial L}{\partial q^1_1}\\
\vdots\\
\frac{\partial L}{\partial q^1_p}\\
-L+q^1_\jmath\frac{\partial L}{\partial q^1_\jmath}\\
0\\
0\\
\vdots\\
0\\
\end{array}
\right)
+\lambda_2\left(
\begin{array}
l\frac{\partial L}{\partial q^2_1}\\
\vdots\\
\frac{\partial L}{\partial q^2_p}\\
0\\
-L+q^2_\jmath\frac{\partial L}{\partial q^2_\jmath}\\
0\\
\vdots\\
0\\
\end{array}
\right)
+...+ \lambda_{n-p}\left(
\begin{array}
l\frac{\partial L}{\partial q^{n-p}_1}\\
\vdots\\
\frac{\partial L}{\partial q^{n-p}_p}\\
0\\
0\\
\vdots\\
0\\
-L+q^{n-p}_\jmath\frac{\partial L}{\partial q^{n-p}_\jmath}\\
\end{array}
\right)
\]
$=\lambda_1v^1+\lambda_2v^2+...+\lambda_{n-p}v^{n-p}$, so the theorem is
proved.\\

\end{proof}
\begin{exm}

In the case $n=3$ and $p=2$, the functional area is
$L(x,y,z,p,q)=\sqrt{1+p^2+q^2}$  then we find
\[
X=\left(\frac{p}{\sqrt{1+p^2+q^2}},\frac{q}{\sqrt{1+p^2+q^2}},\frac{p^2}{\sqrt{1+p^2+q^2}}+\frac{q^2}{\sqrt{1+p^2+q^2}}-\sqrt{1+p^2+q^2}\right)
\]
\[
=\frac{1}{\sqrt{1+p^2+q^2}}(p,q,-1)
\]
which have the same direction as the normal $n$ to the hypersurface $\Sigma$
spanned by $(1,0,p)$ and $(1,0,q)$ in the classical euclidean sens.
\end{exm}

\begin{exm}
If we take $n=4$ and $p=2$ then $f:\mathbb{R}^2\rightarrow\mathbb{R}^2$, $\Sigma_t$
are a domains with boundary of dimension 2 of $\mathbb{R}^4$, we define the
functional area by
$L(x,y,f_1,f_2,q^1_1,q^1_2,q^2_1,q^2_2):=\sqrt{(q^1_1)^2+(q^1_2)^2+(q^2_1)^2+(q^2_2)^2+(q^1_1q^2_2-q^1_2q^2_1)^2}$,
an easy calculation gives us that the normal subspace to $\mathcal{N}_x$ is
$\mathcal{V}=(v^1,v^2)$ with
\[
v^1=
\left(
\begin{array}
l
\frac{q^1_1+q^2_2(q^1_1q^2_2-q^1_2q^2_1)}{\sqrt{(q^1_1)^2+(q^1_2)^2+(q^2_1)^2+(q^2_2)^2+(q^1_1q^2_2-q^1_2q^2_1)^2}}\\

\frac{q^1_2-q^2_1(q^1_1q^2_2-q^1_2q^2_1)}{\sqrt{(q^1_1)^2+(q^1_2)^2+(q^2_1)^2+(q^2_2)^2+(q^1_1q^2_2-q^1_2q^2_1)^2}}\\

-\frac{(q^2_2)^2+(q^2_1)^2}{\sqrt{(q^1_1)^2+(q^1_2)^2+(q^2_1)^2+(q^2_2)^2+(q^1_1q^2_2-q^1_2q^2_1)^2}}\\

\ \ \ \ 0\\
\end{array}
\right),
v^2=
\left(
\begin{array}
l
\frac{q^2_1-q^1_2(q^1_1q^2_2-q^1_2q^2_1)}{\sqrt{(q^1_1)^2+(q^1_2)^2+(q^2_1)^2+(q^2_2)^2+(q^1_1q^2_2-q^1_2q^2_1)^2}}\\

\frac{q^2_2+q^1_1(q^1_1q^2_2-q^1_2q^2_1)}{\sqrt{(q^1_1)^2+(q^1_2)^2+(q^2_1)^2+(q^2_2)^2+(q^1_1q^2_2-q^1_2q^2_1)^2}}\\

\ \ \ \ 0\\
-\frac{(q^1_1)^2+(q^1_2)^2}{\sqrt{(q^1_1)^2+(q^1_2)^2+(q^2_1)^2+(q^2_2)^2+(q^1_1q^2_2-q^1_2q^2_1)^2}}\\

\end{array}
\right)
\]
then
\[
\mathcal{V}=\frac{1}{L}\left(\left(
\begin{array}
l q^1_1+q^2_2(q^1_1q^2_2-q^1_2q^2_1)\\
q^1_2-q^2_1(q^1_1q^2_2-q^1_2q^2_1)\\
-(q^2_2)^2-(q^2_1)^2\\
\ \ \ \ 0\\
\end{array}
\right),
\left(
\begin{array}
l q^2_1-q^1_2(q^1_1q^2_2-q^1_2q^2_1)\\
q^2_2+q^1_1(q^1_1q^2_2-q^1_2q^2_1)\\
\ \ \ \ 0\\
-(q^1_1)^2-(q^1_2)^2\\
\end{array}
\right)\right)
\]

\end{exm}
\begin{rem}
The tangent space to $\Sigma_M$ in neighborhood of a point $M$ is spanned by
\[
T\Sigma_M=\langle\left(
\begin{array}{c}
1\\
0\\
q^1_1\\
q^2_1
\end{array}
\right), \ \
\left(
\begin{array}{c}
0\\
1\\
q^1_2\\
q^2_2
\end{array}
\right)\rangle
\]
Then, in the Euclidean case, the subspace orthogonal to $\Sigma_M$ is spanned by

\[
(T\Sigma_M)^\bot=\langle\left(
\begin{array}{c}
q^1_1\\
q^1_2\\
-1\\
0
\end{array}
\right), \ \
\left(
\begin{array}{c}
q^2_1\\
q^2_2\\
0\\
-1
\end{array}
\right)\rangle
\]
which coincides with our result when
\[
\left\{
\begin{array}{c}
(q^1_1)^2+(q^1_2)^2=(q^2_2)^2+(q^2_1)^2=1\\
\\
q^1_1q^2_2-q^1_2q^2_1=0 \ \ \ \ \ \ \ \ \ \ \ \ \ \ \ \ \ \
\end{array}
\right.
\]
In other form
\[
\left(
\begin{array}{cc}
q^1_1&q^1_2\\
\\
q^2_1&q^2_2
\end{array}
\right)=
\left(
\begin{array}{cc}
\cos\theta&\sin\theta\\
\\
\cos\theta&\sin\theta\\
\end{array}
\right)
\text{ or }
\left(
\begin{array}{cc}
q^1_1&q^1_2\\
\\
q^2_1&q^2_2
\end{array}
\right)=
\left(
\begin{array}{cc}
-\cos\theta&\sin\theta\\
\\
-\cos\theta&\sin\theta\\
\end{array}
\right)
\]
So we can see the difference between our space and an euclidean space.
\end{rem}
\section{Determination of the normal unit vector to a hypersurface}

\begin{thm}
In the space of $n$ dimension the determinant of Gram calculates the volume
$\mathcal{V}$ of a parallelepiped formed by $n$ vectors $\xi_1,...,\xi_n$ with
$\xi_\imath=\xi_\imath^\jmath e_\jmath$ for $\imath, \jmath=1,...,n$ if the espace is equiped by a metric $g$, then
\[
\mathcal{V}=\sqrt{G(\xi_1,...,\xi_n)}.
\]
where
\[
G(\xi_1,..., \xi_n):=
\left\vert
\begin{array}{ccc}
g(\xi_1,\xi_1)&\ldots&g(\xi_1,\xi_n)\\
\vdots&\ddots&\vdots\\
g(\xi_n,\xi_1)&\ldots&g(\xi_n,\xi_n)\\
\end{array}
\right\vert
\]

\end{thm}
\begin{proof}
The proof of this theorem is made by reccurence on $n$.\\
If $n=1$ obvious case, assuming the property is true for any family of $n$
vectors, and prove to $n+1$. The volume $\upsilon$ of the parallelepiped $n+1$
dimension is by defintion the volume of the basis F, space generated by the
first $n$ vectors of volume equal  $\sqrt{G(\xi_1,...,\xi_n)}$, by assumption of
recurrence, multiplied by the height of $\xi_{n+1}$. Then
$\upsilon=\sqrt{G(\xi_1,...,\xi_{n+1})}$ due to the third point of the following
consequence.
\end{proof}

\begin{cons}
\begin{enumerate}
\item The volume constructed on $n$ vectors of the same origin is equal to the
determinant formed by the components of these vectors, multiplied by the volume
constructed by $n$ unit vectors.
\[
\mathcal{V}(\xi_1,...,\xi_n)=\mathcal{V}\left(\frac{\xi_1}{\mid\mid\xi_1\mid\mid},...,\frac{\xi_n}{\mid\mid\xi_n\mid\mid}\right)\mid\xi_1,...,\xi_n\mid
\]

\item In the case where the vectors $\xi_\imath$ are unit we have
$g(\xi_\imath,\xi_\jmath)= g_{\imath\jmath}$ then the volume formed by $n$ unit
vectors is:
\[
G(\xi_1,...,\xi_n)=\sqrt{g}
\]
\item Let $v$ the orthogonal vector to parallelepiped formed by $n$ vectors
$\xi_1,...,\xi_n$ then we have:
\[
G(v,\xi_1,...,\xi_n)=\sqrt{g(v,v)}G(\xi_1,...,\xi_n)
\]
\end{enumerate}
\end{cons}

\begin{proof}
2 and 3 are abvious, we proved only 1.\\
We have $g(\xi_\imath,\xi_\jmath)=~^t\xi_\imath g_{\imath\jmath}\xi_\jmath$ if
we suppose $\xi_\imath= (\xi^1_\imath,...,\xi^n_\imath)$ then:
\[
G(\xi_1,...,\xi_n)=\left\vert\left(
\begin{array}{ccc}
\xi^1_1&\ldots&\xi^n_1\\
\vdots&\ddots&\vdots\\
\xi^1_n&\ldots&\xi^n_n\\
\end{array}
\right)\left(
\begin{array}{ccc}
g_{11} &\ldots &g_{1n}\\
\vdots &\ddots &\vdots\\
g_{n1} &\ldots &g_{nn}\\
\end{array}
\right)
\left(
\begin{array}{ccc}
\xi^1_1&\ldots&\xi^1_n\\
\vdots&\ddots&\vdots\\
\xi^n_1&\ldots&\xi^n_n\\
\end{array}
\right)\right\vert
\]
So using 2, we can find 1.
\end{proof}
\begin{defi}
If we denote $(e^\ast_1,...,e^\ast_n)$ the dual basis in the espace $E$ of $n$
dimension, let $n$ vectors $\xi_1,\xi_2,...,\xi_n$
if we considered that $1\leq\imath_1<...<\imath_p\leq n$ we define by
$e^\ast_{\imath_1}\wedge...\wedge e^\ast_{\imath_p}(\xi_1,
\xi_2,...,\xi_p)$ the determinant of the matrix of order $p$ formed by
${\imath_k}^{th}$ component of vectors $\xi_\jmath$ where $\jmath,k=1,...p$ and $\imath_k\in\{1,...n\}$:
\[
e^\ast_{\imath_1}\wedge...\wedge e^\ast_{\imath_p}(\xi_1,\xi_2,...,\xi_p)=
\left\vert
\begin{array}{ccc}
\xi^{\imath_1}_1&\ldots&\xi^{\imath_1}_p\\
\vdots&\ddots&\vdots\\
\xi^{\imath_p}_1&\ldots&\xi^{\imath_p}_p\\
\end{array}
\right\vert.
\]
\end{defi}

\begin{thm}
The length $\ell$ of the normal vector $v$ to the hypersurface $\Sigma$ is
\[
\sqrt{g}.
\]
\end{thm}

\begin{proof}
Note that the subspace generated by $n-1$ vectors $p_\imath=
(p^1_\imath,...,p^n_\imath)$ for $\imath=1,...,n-1$ is the tangent space of the
hypersurface $\Sigma$ and his volume is $d\sigma$ the volume of the parllelepide
of $n$ dimension spanned by $\Sigma$ and $v$ is $\mathcal{V}= \ell d\sigma$, to
simplify the calculation, we introduce the variables $\xi_1,...,\xi_n$ as
$\xi_n=-\frac{\xi_1}{p_1}=...=-\frac{\xi_{n-1}}{p_{n-1}}$ and we denote the
function $F$ as $\xi_nL(x^1,...,x^n;
\frac{\xi_1}{\xi_n},...,-\frac{\xi_{n-1}}{\xi_n})=F(x^1,...,x^n;\xi_1,...\xi_n)$
and like $F$ homogeneous and of degree 1 $\xi_1,...,\xi_n$ then we right (3):
\begin{equation}
v=\left(\frac{\partial F}{\partial \xi_1},...,\frac{\partial F}{\partial
\xi_n}\right)
\end{equation}
secondly according to one of the previous consequence then:
\[
\mathcal{V}=\sqrt{g}\left\vert
\begin{array}{ccc}
\frac{\partial F}{\partial \xi_1}&\ldots&\frac{\partial F}{\partial \xi_n}\\
\xi^1_1&\ldots&\xi^n_1\\
\vdots&\ddots&\vdots\\
\xi^1_{n-1}&\ldots&\xi^n_{n-1}\\
\end{array}
\right\vert
\]
If we denote $(e_1^\ast,...,e_n^\ast)$ the dual basis then the element of
surface $d\sigma =\sum_{\imath=1}^n(-1)^{\imath-1}\frac{\partial F}{\partial
\xi_\imath}e_1^\ast\wedge...\wedge e_{\imath-1}^\ast\wedge
e_{\imath+1}^\ast...\wedge e_n^\ast$, now it remains to calculate:
\[
d\sigma(\xi_1,...,\xi_{n-1})=\sum_{\imath=1}^n\frac{\partial F}{\partial
\xi_\imath} e_1^\ast \wedge...\wedge e_{\imath-1}^\ast\wedge
e_{\imath+1}^\ast...\wedge e_n^\ast (\xi_1,...,\xi_{n-1})
\]
\[
  =\sum_{\imath=1}^n(-1)^{\imath-1}\frac{\partial
F}{\partial\xi_\imath}\left\vert
\begin{array}{ccc}
\xi^1_1&\ldots&\xi^1_{n-1}\\
\vdots&\ddots&\vdots\\
\xi^{\imath-1}_1&\ldots&\xi^{\imath-1}_{n-1}\\
\xi^{\imath+1}_1&\ldots&\xi^{\imath+1}_{n-1}\\
\vdots&\ddots&\vdots\\
\xi^n_1&\ldots&\xi^n_{n-1}\\
\end{array}
\right\vert
=
\left\vert
\begin{array}{ccc}
\frac{\partial F}{\partial\xi_1}&\ldots&\frac{\partial F}{\partial\xi_n}\\
\xi^1_1&\ldots&\xi^n_1\\
\vdots&\ddots&\vdots\\
\xi^1_{n-1}&\ldots&\xi^n_{n-1}\\
\end{array}
\right\vert.
\]
which giving $\ell=\sqrt{g}$.
\end{proof}

\begin{cons}
The components of $\nu$ on dual basis are:
\[
\sqrt{g}\left(\frac{\xi_1}{F},...,\frac{\xi_n}{F}\right).
\]
\end{cons}

\begin{proof}
Denote respectively $\ell^\imath$ and $\ell_\imath$ the components of $\nu$ in
the basis and in the dual basis, then use (4) we have
$\ell^\imath=\frac{1}{\sqrt{g}}\frac{\partial F}{\partial\xi_\imath}$, and like
$\nu$ is normal then $\ell^\imath\ell_\imath=1$, or $F$ is homogeneous of degree
one in $\xi_\imath$ then
\[
\frac{1}{\sqrt{g}}\frac{\partial
F}{\partial\xi_\imath}\xi_\imath=\frac{1}{\sqrt{g}}F
\Rightarrow\frac{1}{\sqrt{g}}\frac{\partial F}{\partial\xi_\imath}\sqrt{g}
\frac{\xi_\imath}{F}=1
\]
which gives the normal component of unit vector in the dual basis.
\end{proof}

\newpage
\nocite{*}
\bibliographystyle{plain}
\bibliography{references}

\end{document}